\documentclass[14pt,a4paper]{article}

\usepackage{mathrsfs}
\usepackage{amscd}
\usepackage{hyperref}
\hypersetup{colorlinks=true,citecolor=green,linkcolor=blue,urlcolor=blue}
\usepackage{tipa}
\usepackage{amssymb}
\usepackage{amsfonts}
\usepackage{stmaryrd}
\usepackage{amssymb}

\usepackage{CJK}  
\usepackage{indentfirst} 

\usepackage{latexsym}   
\usepackage{bm}         

\usepackage{pifont}
\usepackage{bm}
\usepackage{amsmath,amssymb,amsfonts}
\usepackage{indentfirst}

\usepackage{xcolor}
\usepackage{cases}
\usepackage{wasysym}
\usepackage{amsthm}

\newtheorem{theorem}{Theorem}[section]
\newtheorem{lemma}[theorem]{Lemma}
\newtheorem{proposition}[theorem]{Proposition}

\newtheorem{remark}[theorem]{Remark}
\newtheorem{corollary}[theorem]{Corollary}
\newtheorem{definition}[theorem]{Definition}
\CJKtilde   

\begin{document}

\title{\Large \textbf{On certain Iwahori--Hecke modules of $GL_3$ in characteristic $p$}}
\date{}
\author{\textbf{Peng Xu}}
\maketitle

\begin{abstract}
In this note, we show that the natural analogue of certain finiteness result of Barthel--Livn$\acute{\text{e}}$ on $GL_2$ fails for $GL_3$. More precisely, within the pro-$p$-Iwahori invariants of a maximal compact induction of $GL_3$, we show there exist non-zero Iwahori--Hecke submodules of \emph{infinite} codimension.
\end{abstract}

\tableofcontents

\section{Introduction}\label{sec: intro}

In their pioneering work (\cite{B-L94}, \cite{B-L95}), Barthel--Livn$\acute{\text{e}}$ gave a classification of irreducible smooth $\overline{\mathbf{F}}_p$-representations (with central characters) of $GL_2$ over a non-archimedean local field $F$ of residue characteristic $p$. One feature of their work is they proved the existence of Hecke eigenvalues \emph{without the assumption of admissibility}, where a key ingredient in their argument is to show \emph{any} non-zero Iwahori--Hecke submodule of the pro-$p$-Iwahori invariants of a maximal compact induction is \emph{finite codimensional}. The goal of this note is to demonstrate the analogue of such an ingredient \emph{fails} for $G=GL_N (F) (N\geq 3)$.

Let $\mathfrak{o}_F$ be the ring of integers of $F$. We fix a uniformizer $\varpi$ of $F$. Let $K$ be the maximal compact open subgroup $GL_N (\mathfrak{o}_F)$, and $I$ (resp, $I_1$) be the standard (resp, pro-$p$-) Iwahori subgroup of $G$. Let $Z$ be the center of $G$.
\medskip

Let $\sigma$ be an irreducible smooth $\overline{\mathbf{F}}_p$-representation of $K$. Denote also by $\sigma$ its extension to $KZ$ on which $\varpi$ acts trivially. Let $\chi_\sigma$ be the character of $I$ on the line $\sigma^{I_1}$, extended to $IZ$ by requiring $\chi (\varpi)=1$. Our main result is as follows:

\begin{theorem}(Corollary \ref{GL_3 fails})\label{main result: intro}
Assume $N=3$. Then there are non-zero $\mathcal{H} (IZ, \chi_\sigma)$-submodules of $(\textnormal{ind}^G _{KZ} \sigma)^{IZ, \chi_\sigma}$ of infinite codimension.
\end{theorem}

\begin{remark}
This note grows out of an unsuccessful attempt to show the existence of Hecke eigenvalues for $p$-modular representations of $G$. Such a problem is listed as \cite[Question 8]{AHHV17b}, and to our knowledge it is only answered positively for $GL_2$ (\cite{B-L94}) and $U(2, 1)$ (\cite{X2018}).
\end{remark}

\section{Notations and preliminaries}\label{sec: notation and preliminaries}

\subsection{General notations}\label{subsec: notations}
Let $F$ be a non-archimedean local field, with ring of integers $\mathfrak{o}_F$ and maximal ideal $\mathfrak{p}_F$, and let $k_F$ be its residue field of characteristic $p$ and order $q$. Fix a uniformizer $\varpi$ in $F$. Let $G= GL_N (F)$ ($N\geq 3$), $K= GL_N (\mathfrak{o}_F)$, $Z\cong F^\times$ the center of $G$. Let $K_1$ be the kernel of the reduction map $\emph{red}: K\rightarrow GL_N (k_F)$. Denote the latter group by $\mathbb{G}$. Let $\mathbb{B}= \mathbb{T}\ltimes \mathbb{U}$ be the standard Borel subgroup of $\mathbb{G}$, with the subgroup $\mathbb{T}$ of diagonal matrices and the upper unipotent radical $\mathbb{U}$. Let $I$ (resp, $I_1$) be the (resp, pro-$p$-) Iwahori subgroup of $G$, defined as the inverse image of $\mathbb{B}$ (resp, $\mathbb{T}$) in $K$ via $\emph{red}$.

Denote by $\omega_1$ and $\omega_2$ the following elements in $G$:
\begin{center}
$\omega_1= \begin{pmatrix} \beta  & 0  \\ 0  & I_{N-2}
\end{pmatrix}$, $\omega_2= \begin{pmatrix} 0  & I_{N-1}  \\ 1  & 0
\end{pmatrix}$,
\end{center}
where $\beta$ is the $2 \times 2$ matrix $\begin{pmatrix} 0  & 1  \\ 1  & 0 \end{pmatrix}$.

We identify the finite Weyl group $W_0$ of $G$ with the group of permutation matrices. Note that $W_0$ is generated by $\omega_1$ and $\omega_2$.

Put $\gamma= \omega_2\cdot \text{diag}(\varpi, I_{N-1})$. Recall that the normalizer of $I_1$ in $G$ is generated by $I$ and $\gamma$.

\medskip
In this note, all representations are smooth over $\overline{\mathbf{F}}_p$. As $I_1$ is a pro-$p$-sylow subgroup of $I$, any character $\chi$ of $I$ taking values in $\overline{\mathbf{F}}^\times _p$ factors through the finite abelian group $I/I_1\cong \mathbb{T}$. Hence the group $W_0$ acts on the set of characters of $I$, and the conjugate of $\chi$ by an $\omega \in W_0$ is denoted as $\chi^\omega$.

\subsection{Weights}\label{subsec: weights}
Let $\sigma$ be an irreducible smooth representation of $K$. As $K_1$ is pro-$p$ and normal in $K$, $\sigma$ factors through the finite group $\mathbb{G}$, i.e., $\sigma$ is the inflation of an irreducible representation of $\mathbb{G}$. Conversely, any irreducible representation of $\mathbb{G}$ inflates to an irreducible smooth representation of $K$. We may therefore identify irreducible smooth representations of $K$ with irreducible representations of $\mathbb{G}$, and we shall call them \emph{weights} of $K$ or $\mathbb{G}$ from now on.

For a weight $\sigma$ of $K$, it is well-known that $\sigma^{I_{1,K}}$ is one-dimensional (\cite[Theorem 6.12]{C-E2004}).

\section{The Iwahori--Hecke algebra $\mathcal{H}(IZ, \chi)$}

Let $\chi$ be a character of $I$, extended to $IZ$ by requiring $\chi (\varpi)=1$. By \cite[Proposition 5]{B-L94}, the algebra  $\mathcal{H}(IZ, \chi):= \text{End}_G (\text{ind}^G _{IZ} \chi)$ is isomorphic to the convolution algebra $\mathcal{H}_{IZ} (\chi)$ given by:
\begin{center}
$\mathcal{H}_{IZ} (\chi)= \{\varphi: G \rightarrow \overline{\mathbf{F}}_p \mid \varphi(i_1 g i_2)= \chi(i_1 i_2) \varphi(g), \forall i_1, i_2\in IZ, g\in G, \text{smooth~with~compact~support~modulo}~KZ\}$.
\end{center}
Denote by $T_\varphi$ the Hecke operator in $\mathcal{H}(IZ, \chi)$ which corresponds to a function $\varphi\in \mathcal{H}_{IZ} (\chi)$, via the aforementioned isomorphism.

For an element $g\in G$, denote by $\varphi_g$ the function in $\mathcal{H}_{IZ} (\chi)$ supported on $IZ gI$ and satisfying $\varphi(g)=1$. Such condition uniquely determines $\varphi_g$ if it exists. We will write $T_{\varphi_g}$ as $T_g$ for short.

Recall the Iwahori decomposition of $G$:
\begin{equation}\label{Iwahori decom}
G= \bigcup _{\omega\in W_0, \vec{a}\in \mathbb{Z}^{N-1}}~IZ \omega t_{\vec{a}} I
\end{equation}

\begin{lemma}\label{defined condition for var}
There is a non-zero function $\varphi\in \mathcal{H}_{IZ} (\chi)$ supported on $IZ \omega t_{\vec{a}}I$, for some $\omega\in W_0$ and $\vec{a}\in \mathbb{Z}^{N-1}$, if and only if:
\begin{center}
$\chi= \chi^\omega$.
\end{center}
\end{lemma}
\begin{proof}
The only condition on the value of $\varphi$ at $\omega t_{\vec{a}}$ is: for any $i_1, i_2\in I$ satisfying $i_1 \omega t_{\vec{a}}= \omega t_{\vec{a}} i_2$, the identity $\chi(i_1) \varphi(\omega t_{\vec{a}})= \varphi(\omega t_{\vec{a}})\chi(i_2)$ holds. The lemma follows from some simple computation.
\end{proof}

For $\omega \in W_0, \vec{a} \in \mathbb{Z}^{N-1}$, we consider the function $\varphi_{\omega\cdot t_{\vec{a}}}$. Note that $\varphi_{\omega\cdot t_{\vec{a}}}$ is only well-defined under the condition $\chi=\chi^\omega$ (Lemma \ref{defined condition for var}). By Lemma \ref{defined condition for var} and the Iwahori decomposition \eqref{Iwahori decom}, the algebra $\mathcal{H}_{IZ}(\chi)$ has a basis $\{\varphi_{\omega\cdot t_{\vec{a}}}\mid \omega\in W(\chi), \vec{a}\in \mathbb{Z}^{N-1}\}$, where $W(\chi) := \{\omega\in W_0 \mid  \chi=\chi^\omega\}$. As mentioned above, we will write $T_{\varphi_{\omega\cdot t_{\vec{a}}}}$ as $T_{\omega\cdot t_{\vec{a}}}$ for short.

\begin{corollary}\label{basis of I-W}
The set $\{T_{\omega\cdot t_{\vec{a}}}\mid \omega\in W(\chi), \vec{a}\in \mathbb{Z}^{N-1}\}$ consists of a basis of the algebra $\mathcal{H}(IZ, \chi)$.
\end{corollary}

\subsection{The Iwahori--Hecke case}
We consider the Iwahori--Hecke case first (\cite{Vig2005}, \cite{Oll06}).

\begin{proposition}\label{structure of Iwahori-Hecke: degenerate case}
When $W(\chi)=W_0$, the algebra $\mathcal{H}(IZ, \chi)$ is non-commutative and generated by the operators $T_\gamma$ and $T_{\omega_1}$. More precisely, there is an isomorphism of algebras:
\begin{center}
$\mathcal{H}(IZ, \chi) \cong \overline{\mathbf{F}}_p [T_\gamma, T_{\omega_1}]/(T^N _\gamma -1, T^2 _{\omega_1} + T_{\omega_1})$.
\end{center}
\end{proposition}

\begin{remark}
Proposition \ref{structure of Iwahori-Hecke: degenerate case} generalizes \cite[Proposition 11]{B-L94}. Note that the presentation here is different from that in loc.cit, as a different Hecke operator is chosen.
\end{remark}

\subsection{The semi-regular case for $N=3$}

We assume $N=3$.

In this case, we assume $\chi$ is of the form $1\otimes 1\otimes \eta$ for some non-trivial character $\eta$ of $F^\times$, so the group $W(\chi)$ is just $\{Id, \omega_1\}$. By \cite[Proposition 1]{Oll06}, that is essentially the only semi-regular case.

\begin{proposition}(\cite[Theorem 25]{Oll06})\label{structure of Iwahori-Hecke: semi-regular}
The Hecke algebra $\mathcal{H}(IZ, \chi)$ is generated by $T_{\omega_1}, T_{\omega_1 \cdot t_{(1, 0)}}, T_{\omega_1 \cdot t_{(0, -1)}}$. More precisely, the algebra $\mathcal{H}(IZ, \chi)$ is isomorphic to:
\begin{center}
$\overline{\mathbf{F}}_p [T_{\omega_1}, T_{\omega_1 \cdot t_{(1, 0)}}, T_{\omega_1 \cdot t_{(0, -1)}}]/ (T^2 _{\omega_1} +T_{\omega_1}, T_{\omega_1 \cdot t_{(1, 0)}}\cdot T_{\omega_1 \cdot t_{(0, -1)}}, T_{\omega_1 \cdot t_{(0, -1)}}\cdot T_{\omega_1 \cdot t_{(1, 0)}})$.
\end{center}
\end{proposition}

\subsection{The regular case for $N=3$}

In this case, the group $W(\chi)$ is trivial. The algebra $\mathcal{H}(IZ, \chi)$ has a linear basis $\{T_{t_{\vec{a}}}\mid \vec{a} \in \mathbb{Z}^{N-1}\}$ (Corollary \ref{basis of I-W}).
\begin{proposition}(\cite[Proposition 18]{Oll06})\label{structure of Iwahori-Hecke: regular}
The Hecke algebra $\mathcal{H}(IZ, \chi)$ is commutative and generated by the following six Hecke operators
\begin{center}
$T_{t_{(1, 0)}}, T_{t_{(-1, 0)}}, T_{t_{(0, 1)}}, T_{t_{(0, -1)}}, T_{t_{(1, 1)}}, T_{t_{(-1, -1)}}$
\end{center}
with nine relations:
\begin{center}
$T_{t_{(1, 0)}} \cdot T_{t_{(-1, 0)}}=0, T_{t_{(1, 0)}} \cdot T_{t_{(0, 1)}}=0, T_{t_{(1, 0)}} \cdot T_{t_{(-1, -1)}}=0,$

$T_{t_{(0, 1)}} \cdot T_{t_{(0, -1)}}=0, T_{t_{(0, 1)}} \cdot T_{t_{(-1, -1)}}=0, T_{t_{(1, 1)}} \cdot T_{t_{(-1, -1)}}=0,$

$T_{t_{(1, 1)}} \cdot T_{t_{(0, -1)}}=0, T_{t_{(1, 1)}} \cdot T_{t_{(-1, 0)}}=0, T_{t_{(-1, 0)}} \cdot T_{t_{(0, -1)}}=0$.
\end{center}
\end{proposition}

\section{The pro-$p$-Iwahori invariants of $\text{ind}^G _{KZ} \sigma$}\label{sec: I_1-invariants}
Let $\sigma$ be a weight of $K$. We extend $\sigma$ to a representation of $KZ$ by requiring $\varpi$ to act trivially.  Let $\text{ind}_{KZ} ^G \sigma$ be the smooth representation compactly induced from $\sigma$, i.e., the representation of $G$ with underlying space $S(G, \sigma)$
\begin{center}
$S(G, \sigma)=\{f: G\rightarrow \sigma\mid  f(kg)=\sigma (k)\cdot f(g), \forall~k\in KZ, g\in G,~\text{smooth~with~supp~compact~modulo}~KZ\}$
\end{center}
and $G$ acting by right translation.

In this section, we determine a basis of the $I_1$-invariants of a maximal compact induction $\text{ind}^G _{KZ} \sigma$. Compare with \cite{Oll2015}.

\subsection{The Iwasawa--Iwahori decomposition of $G$}\label{subsec: Iwasawa--Iwahori decomposition}

Recall the following Iwasawa--Iwahori decomposition of $G$:

\begin{equation}\label{Iwahori-Hecke decom}
G= \bigcup _{\vec{a}\in \mathbb{Z}^{N-1}}~KZ \cdot t_{\vec{a}}\cdot I_1 ,
\end{equation}
where $t_{\vec{a}}$ denotes the diagonal matrix:
\begin{center}
$t_{\vec{a}}= \text{diag}(\varpi^{a_1}, \ldots, \varpi^{a_{n-1}}, 1)$.
\end{center}

By \eqref{Iwahori-Hecke decom}, a function $f$ in $\text{ind}^G _{KZ}\sigma$ invariant under the action of $I_1$ is uniquely determined by its values on all diagonal matrices $t_{\vec{a}}$ with $\vec{a}\in \mathbb{Z}^{N-1}$.

\subsection{A decomposition of $\mathbb{Z}^{N-1}$}\label{subsec: decomposition of Z^n-1}
\begin{definition}
For an $\omega\in W_0$, denote by $S_\omega$ the subset of $\mathbb{Z}^{N-1}$ given by:
\begin{center}
$S_\omega= \{\vec{a}\in \mathbb{Z}^{N-1} \mid (\omega t_{\vec{a}}\cdot I \cdot (\omega t_{\vec{a}})^{-1})\cap K \subseteq I \}$.
\end{center}
\end{definition}

\begin{lemma}\label{combin lemma}
We have:

$(1)$.~For any $\vec{a}\in S_\omega$, we have
\begin{center}
$I= K_1 \cdot (\omega t_{\vec{a}}\cdot I \cdot (\omega t_{\vec{a}})^{-1}\cap K)$.
\end{center}

$(2)$.~There is a disjoint decomposition:
\begin{center}
$\mathbb{Z}^{N-1}= \bigcup_{\omega\in W_0} S_\omega$.
\end{center}

\end{lemma}

\begin{proof}
Denote by $I(\omega, \vec{a})$ the group $\omega t_{\vec{a}}\cdot I \cdot (\omega t_{\vec{a}})^{-1}$. By definition, if $\vec{a} \in S_\omega$ then $I(\omega, \vec{a}) \cap K$ is contained in $I$.

We start by some general remarks. For a pair $(i, j)$, all $(i, j)$-entries of $I(\omega, \vec{a})$ consist of $\mathfrak{p}^{l^\omega _{ij}} _F$, for some $l^\omega _{ij} \in \mathbb{Z}$. We view $l^\omega _{ij}$ as functions of $\vec{a}$. Some simple observations are given first:

a).~$l^\omega _{ij} = 1- l^\omega _{ji}$ (for $(i, j)$ with $i\neq j$).

b).~If $(i, j)\neq (i', j')$, then $l^\omega _{ij} \neq  l^\omega _{i'j'}$ (as functions of $\vec{a}$).

c).~ $S_\omega= \{\vec{a}\in \mathbb{Z}^{n-1}\mid l^\omega _{ij} (\vec{a})\geq 1~\text{for}~(i, j)~\text{with}~i> j\}$.

\medskip
For $(1)$, it suffices to check that $B\cap K \subseteq I(\omega, \vec{a})\cap K$ for $\vec{a} \in S_\omega$, as we know $I= K_1 (B\cap K)$.  By $a)$ and $c)$ above we have $l^\omega _{ji}= 1- l^\omega _{ij} \leq 0$ for all $(j, i)$ with $j< i$, as required.

For $(2)$, \emph{we show first the decomposition is disjoint.} Note that if $\omega \neq \omega'$, then there exist one pair $(i, j)$ with $i< j$ and another pair $(i', j')$ with $i'< j'$, such that:
\begin{center}
$l^\omega _{ij} = l^{\omega'}_{j'i'}$,
\end{center}
from which and $a), c)$ above we see $S_\omega \cap S_{\omega'} =\varnothing$.

Before moving on, we record the set $S_\omega$ for a few special $\omega$:
\begin{center}
$S_{Id}= \{\vec{a}\in \mathbb{Z}^{n-1}\mid a_1 \leq a_2 \leq ... \leq a_{n-1} \leq 0\}$,

$S_{\omega_1}= \{\vec{a}\in \mathbb{Z}^{n-1}\mid a_2 +1 \leq a_1 \leq a_3 \leq ... \leq a_{n-1}\leq 0\}$,

$S_{\omega_2}= \{\vec{a}\in \mathbb{Z}^{n-1}\mid a_2 \leq a_3 \leq... \leq a_{n-1}\leq 0 < a_1\}$.

\end{center}

\emph{Now we prove $S_\omega \neq \varnothing$ for any $w \in W_0$.}

We claim first that $S_\omega \neq \varnothing$ implies $S_{\omega\omega'} \neq \varnothing$ for any $\omega' \in \langle \omega_2\rangle$. The observation here is for such an $\omega'$ one may find a diagonal matrix $\tilde{t}$ of the form $\text{diag} (\varpi^{a_1}, ..., \varpi^{a_n})$ so that $\omega' \tilde{t}$ normalizes $I$. More precisely, assume we already have:
\begin{equation}\label{inclusion giving S_omega}
t_{\vec{a}} I t^{-1} _{\vec{a}} \cap K \subset \omega^{-1}I\omega
\end{equation}
for some $t_{\vec{a}}$. Conjugating the above by $\omega'$ we get:
\begin{center}
$\omega'^{-1}t_{\vec{a}} I t^{-1} _{\vec{a}}\omega' \cap K \subset \omega'^{-1}\omega^{-1}I\omega\omega'$.
\end{center}
Here, the left side term can be rewritten as
\begin{center}
$\omega'^{-1}t_{\vec{a}} \omega' \tilde{t} (\tilde{t}^{-1}\omega'^{-1} I \omega'\tilde{t}) \tilde{t}^{-1}\omega'^{-1}t_{\vec{a}}^{-1} \omega' \cap K$.
\end{center}
Note that $\omega'^{-1}t_{\vec{a}} \omega' \tilde{t}$ is also a diagonal matrix of the form (uniquely) $\varpi^l \cdot t_{\vec{b}}$ for some $l \in \mathbb{Z}$ and $t_{\vec{b}}$. In all we get
\begin{center}
$t_{\vec{b}} I t^{-1} _{\vec{b}} \cap K  \subset \omega'^{-1}\omega^{-1}I\omega\omega'$
\end{center}
and conclude that $\vec{b} \in S_{\omega \omega'}$.

Secondly, a parallel claim holds for $\omega_1$: if $\vec{a} \in S_\omega $ such that $a_1 \neq a_2$, then $\omega_1 \cdot \vec{a}= (a_2, a_1, ..., a_{N-1}) \in S_{\omega \omega_1} $.

$1)$ $a_1 > a_2$. The condition $\vec{a} \in S_\omega$ means that
\begin{center}
$t_{\vec{a}} I t^{-1} _{\vec{a}} \cap K \subset \omega^{-1} I \omega$.
\end{center}
So by conjugating we have
\begin{center}
$\omega_1 t_{\vec{a}}\omega_1 (\omega_1 I \omega_1) \omega_1 t^{-1} _{\vec{a}} \omega_1 \cap K \subset (\omega_1\omega)^{-1} I \omega \omega_1$.
\end{center}
Note that the group $I$ and $\omega_1 I \omega_1$ are only different at $(12)$ and $(21)$ entries. The assumption here implies the left side group above, i.e., $\omega_1 I(Id, \vec{a})\omega_1 \cap K$, strictly contains
\begin{center}
$\omega_1 t_{\vec{a}}\omega_1 I \omega_1 t^{-1} _{\vec{a}} \omega_1 \cap K$.
\end{center}
Combined with the containing the claim in this case follows.

$2)$.~$a_1 < a_2$. By conjugating we have the same:
\begin{center}
$\omega_1 t_{\vec{a}}\omega_1 (\omega_1 I \omega_1) \omega_1 t^{-1} _{\vec{a}} \omega_1 \cap K \subset (\omega_1\omega)^{-1} I \omega \omega_1$.
\end{center}
The assumption now implies the $(21)$ (resp., $(12)$)-entries of $\omega_1 I(Id, \vec{a})\omega_1$ is $\mathfrak{o}_F$ (resp., $\mathfrak{p}^{a_2 -a_1 +1}_F$). By the containing above, the $(21)$-entries of $(\omega_1\omega)^{-1} I \omega \omega_1$ is also $\mathfrak{o}_F$, and thus its $(12)$-entries gives us $\mathfrak{p}_F$. However,  the $(12)$ (resp., $(21)$)-entries of $I(Id, \omega_1 t_{\vec{a}}\omega_1)$ is $\mathfrak{p}^{a_2 -a_1}_F$ (resp., $\mathfrak{p}_F$). This is enough to see (note that $a_1 < a_2$)
\begin{center}
$\omega_1 t_{\vec{a}}\omega_1 I \omega_1 t^{-1} _{\vec{a}} \omega_1 \cap K \subset (\omega_1\omega)^{-1} I \omega \omega_1$,
\end{center}
as required by the claim.

\medskip

Recall that the group $W_0$ is generated by $\omega_1$ and $\omega_2$. Starting by the special case $\omega= Id$ above (which clearly contains an $\vec{a}$ such that $a_1 \neq a_2$), the assertion $S_\omega \neq \varnothing$ for any $\omega \in W_0$ follows from a simple induction argument and the two claims we have just verified.

\medskip

\emph{Finally, we show any $\vec{a} \in \mathbb{Z}^{n-1}$ belongs to some $S_\omega$}. We find an $\omega$ such that
\begin{center}
$\omega t_{\vec{a}} \omega^{-1}= \varpi^l t_{\vec{b}}$
\end{center}
for some $l \in \mathbb{Z}$ and $\vec{b}\in S_{Id}$. It is easy to see $l$ and $\vec{b}$ are unique, but in general it is not true for $\omega$. We will generate the right one below.

Put $a_N= 0$. Recall the diagonal matrix $t_{\vec{a}}$ is
\begin{center}
$(\varpi^{a_1}, \varpi^{a_2}, ..., \varpi^{a_N}).$
\end{center}
We write $t_{\vec{a}}$ as $t_{(0)}= (\varpi^{a_{0,1}}, \varpi^{a_{0,2}}, ..., \varpi^{a_{0,N}})$.

For step $1$, we conjugate $t_{(0)}$ by a unique $\omega_{(1)}$ to swap $\varpi^{a_{0, j_0}}$ with $\varpi^{a_{0, N}}$ for the $j_0$ such that $a_{0, j_0} =\text{max} \{a_{0, m} \mid 1\leq m \leq N\}$ and $j_0$ is the largest index satisfying that. We obtain a diagonal matrix $t_{(1)}$.

Suppose step $i$ is done, and we get a diagonal matrix $t_{(i)}= (\varpi^{a_{i, 1}}, \varpi^{a_{i, 2}}, ..., \varpi^{a_{i, N}})$. For step $i+1$, we conjugate $t_{(i)}$ by a unique $\omega_{(i+1)}$ to swap $\varpi^{a_{i, j_i}}$ with $\varpi^{a_{i, N-i}}$ for the $j_{i}$ such that $a_{i, j_i}=\text{max}\{a_{i, m} \mid 1\leq m \leq N-i\}$ and $j_i$ is the largest index satisfying that. We obtain a diagonal matrix $t_{(i+1)}$. Note that if $j_i = N- i$, we have nothing to do and $\omega_{(i+1)} =Id$.

We get the desired $\omega$ as the product $\omega_{(N-1)}\cdot \omega_{(N-2)}... \cdot \omega_{(1)}$.

As $\vec{b} \in S_{Id}$, we get
\begin{equation}\label{containing for a in S_omega}
\omega t_{\vec{a}} \omega^{-1} I \omega t^{-1} _{\vec{a}} \omega^{-1} \cap K \subset I.
\end{equation}
We show that $\vec{a} \in S_{\omega}$ for the element $\omega$ generated above, i.e., we have
\begin{center}
$\omega t_{\vec{a}}  I t^{-1} _{\vec{a}}\omega^{-1}  \cap K \subset I$.
\end{center}

It suffices to compare the group $I$ with $\omega^{-1}I \omega$. We consider the pairs $(j, i)$ with $j< i$ so that conjugating by $\omega^{-1}$ swaps the $(ij)$-entry of a matrix with its $(j'i')$-entry for some $j'< i'$. So $(j', i')$ is also a such pair. We must have $a_i \leq a_j$ for every such pair. Indeed, if $a_i > a_j$ happens, the $(ij)$-entries of the left side of \eqref{containing for a in S_omega} is $\mathfrak{o}_F$. This contradicts \eqref{containing for a in S_omega}. If we have further $a_i < a_j$ for every such pair, we are done. This is because the $(ij)$-entries of $I(\omega, \vec{a})\cap K$ is $\mathfrak{p}^{a_{j'} -a_{i'}} _F$. However, the situation $a_j= a_i$ for some pair $(j, i)$ with $j < i$ is already excluded from the process we find $\omega$.

\medskip

The argument of the Lemma is completed.
\end{proof}

For our later purpose, we display the Lemma explicitly for $N=3$.
\begin{center}
$\omega= Id= \begin{pmatrix} 1  & 0 & 0 \\ 0  & 1 & 0\\
0 & 0 & 1
\end{pmatrix}, S_\omega=\{(a_1, a_2)\in \mathbb{Z}^2 \mid a_1 \leq a_2 \leq 0\}$;

$\omega= \omega_1= \begin{pmatrix} 0  & 1 & 0 \\ 1  & 0 & 0\\
0 & 0 & 1
\end{pmatrix}, S_\omega=\{(a_1, a_2)\in \mathbb{Z}^2 \mid a_2 +1 \leq a_1 \leq 0\}$;

$\omega=  \omega_1 \omega_2= \begin{pmatrix} 0  & 0 & 1 \\ 0  & 1 & 0\\
1 & 0 & 0
\end{pmatrix}, S_\omega=\{(a_1, a_2)\in \mathbb{Z}^2 \mid  a_1 \geq a_2 +1 \geq 2\}$;

$\omega= \omega_2 \omega_1= \begin{pmatrix} 1  & 0 & 0 \\ 0  & 0 & 1\\
0 & 1 & 0
\end{pmatrix}, S_\omega=\{(a_1, a_2)\in \mathbb{Z}^2 \mid a_1 \leq 0, a_2 \geq 1\}$;

$\omega=\omega_2= \begin{pmatrix} 0  & 1 & 0 \\ 0  & 0 & 1\\
1 & 0 & 0
\end{pmatrix}, S_\omega=\{(a_1, a_2)\in \mathbb{Z}^2 \mid a_1 \geq 1, a_2 \leq 0\}$;

$\omega= \omega^2 _2= \begin{pmatrix} 0  & 0 & 1 \\ 1  & 0 & 0\\
0 & 1 & 0
\end{pmatrix}, S_\omega=\{(a_1, a_2)\in \mathbb{Z}^2 \mid a_2 \geq a_1 \geq 1\}$.
\end{center}

\medskip

\subsection{A basis of $(\textnormal{ind}^{G} _{KZ}\sigma)^{I_1}$}\label{subsec: basis of I_1 of maximal compact induction}
We fix a non-zero vector $v_0$ in the line $\sigma^{I_1}$. The Iwahori subgroup $I$ acts as a character on $\sigma^{I_1}$, for which we denote it by $\chi_\sigma$.

By $(2)$ of Lemma \ref{combin lemma}, the Iwasawa--Iwahori decomposition \eqref{Iwahori-Hecke decom} is re-written as
\begin{center}
$G= \bigcup_{\omega\in W_0} \bigcup_{\vec{a}\in S_\omega} KZ\cdot t_{\vec{a}} I_1$.
\end{center}

For an $\omega\in W_0$, and an $\vec{a}\in S_\omega$, denote by $f_{\omega, \vec{a}}$ the $I_1$-invariant function in $\text{ind}^{G} _{KZ}\sigma$, supported on the double coset $KZ t_{\vec{a}}I_1$ and having value $v_0$ at the element $\omega t_{\vec{a}}$. Note that $f_{\omega, \vec{a}} (t_{\vec{a}})= \omega^{-1}v_0$.

\begin{proposition}\label{the pro-p iwahori invariants}

The set of functions $\{f_{\omega, \vec{a}}\mid  \omega\in W_0, \vec{a}\in S_\omega\}$ consists of a basis of the $I_1$-invariants of $\textnormal{ind}^{G} _{KZ}\sigma$.

\end{proposition}

\begin{proof}
Let $f$ be an $I_1$-invariant function in $\textnormal{ind}^{G} _{KZ}\sigma$, supported on the coset $KZ t_{\vec{a}}I_1$, for some $\omega\in W_0$ and $\vec{a}\in S_\omega$. For $g\in K\cap \omega t_{\vec{a}}\cdot I_1 \cdot (\omega t_{\vec{a}})^{-1}$, we have $f(g\cdot\omega t_{\vec{a}})= \sigma(g)f(\omega t_{\vec{a}})= f(\omega t_{\vec{a}})$. As $K_1$ acts trivially on $\sigma$, by $(1)$ of Lemma \ref{combin lemma} we see $f(\omega t_{\vec{a}})$ is indeed $I_1$-invariant, and thus $f$ is proportional to $f_{\omega, \vec{a}}$. By the remarks in subsection \ref{subsec: Iwasawa--Iwahori decomposition} , the proposition follows.
\end{proof}

\section{The $\mathcal{H}(IZ, \chi_\sigma)$-module $(\textnormal{ind}^G _{KZ} \sigma)^{IZ, \chi_\sigma}$}

In this part, we compute the right action of an Iwahori--Hecke algebra $\mathcal{H}(IZ, \chi)$ on the $(IZ, \chi)$-isotypic of a maximal compact induction.

\smallskip

\begin{proposition}\label{the action of Iwahori}
For an $\omega\in W_0$ and $\vec{a}\in S_\omega$, we have:

$(1)$.~~The group $I$ acts on the function $f_{\omega, \vec{a}}$ as the character $\chi^{\omega}_\sigma$.

$(2)$.~~ One has
\begin{center}
$\gamma\cdot f_{\omega, \vec{a}}= f_{\omega\cdot \omega^{-1}_2, \vec{a}^\gamma}$,
\end{center}
where
\begin{center}
$\vec{a}^\gamma := (1+a_2 -a_1, 1+a_3- a_1 \cdots, 1+a_{N-1}-a_1, 1-a_1)$.
\end{center}

\end{proposition}

\begin{proof}
For $(1)$, it is a simple computation by definition.

For $(2)$, one displays as follows:
\begin{align*}
t_{\vec{b}}\cdot\gamma &= \text{diag}(\varpi^{b_1}, \cdots, \varpi^{b_{N-1}}, 1)\omega_2 \text{diag}(\varpi, I_{N-1})\\
&= \omega_2 \text{diag}(1, \varpi^{b_1}, \cdots, \varpi^{b_{N-1}})\text{diag}(\varpi,  I_{N-1})\\
&=\omega_2 \varpi^{b_{N-1}}\text{diag}(\varpi^{-b_{N-1}}, \varpi^{b_1 -b_{N-1}}, \cdots, \varpi^{b_{N-2} -b_{N-1}}, 1)\text{diag}(\varpi, I_{N-1})\\
&=\omega_2 \varpi^{b_{N-1}}\text{diag} (\varpi^{1-b_{N-1}}, \varpi^{b_1 -b_{N-1}}, \cdots, \varpi^{b_{N-2} -b_{N-1}}, 1).
\end{align*}
As $f_{\omega, \vec{a}}$ is supported on $KZ t_{\vec{a}}I_1$, one concludes that the function $\gamma f_{\omega, \vec{a}}$ will be supported on a single coset $KZ t_{\vec{b}}I_1$, where $\vec{b}$ is that given in the statement. By definition, we check that $\gamma f_{\omega, \vec{a}} (\omega' t_{\vec{b}})= v_0$, where $\omega'= \omega \omega^{-1}_2$. Hence the claim.
\end{proof}

By $(1)$ of last Proposition, we have
\begin{corollary}\label{basis of isotypic}
For an $\omega\in W_0$, a basis of the $(IZ, \chi^\omega _\sigma)$-isotypic of $\textnormal{ind}^G _{KZ} \sigma$ is given by:

\begin{center}
$\{f_{\nu, \vec{a}} \mid \nu\in \omega\cdot W(\chi_\sigma), \vec{a}\in S_{\nu}\}.$
\end{center}

\end{corollary}

\subsection{The Iwahori--Hecke case}

In this case, the group $W(\chi_\sigma)= W_0$. By Corollary \ref{basis of I-W}, we know $(\textnormal{ind}^G _{KZ} \sigma)^{IZ, \chi_\sigma}$ is the whole space $(\textnormal{ind}^G _{KZ} \sigma)^{I_1}$, and the latter has a basis $\{f_{\omega, \vec{a}} \mid \omega \in W_0, \vec{a} \in S_\omega\}$ (Proposition \ref{the pro-p iwahori invariants}).

\medskip

The following proposition generalizes \cite[Lemma 14, 15]{B-L95} and \cite[Proposition 17]{B-L94} to $GL_N (N\geq 3)$.
\begin{proposition}\label{action of IH: trivial case}
Assume that $W(\chi_\sigma)= W_0$.

$(1)$.~We have $f_{\omega, \vec{a}}\mid T_\gamma= f_{\omega\omega_2, \gamma\cdot\vec{a}}$, where
\begin{center}

$\gamma\cdot\vec{a} :=(1-a_{N-1}, a_1 -a_{N-1}, \cdots, a_{N-2}-a_{N-1})$.

\end{center}

$(2)$.~We have $f_{\omega, \vec{a}}\mid T_{\omega_1}= \begin{cases} f_{\omega \omega_1, \omega_1 \cdot\vec{a}}, ~~~\text{if}~a_1 >a_2; \\

c_{\sigma, \omega} f_{\omega, \vec{a}}, ~~~~\text{if}~N=3,~a_1= a_2; \\

-f_{\omega, \vec{a}}, ~~~~~~~~\text{if}~N=3,~a_1< a_2. \\
\end{cases}$\\
Here, $\omega_1 \cdot\vec{a} := (a_2, a_1, ..., a_{N-1})$, and $c_{\sigma, \omega}$ is a constant related to $\sigma$.

\end{proposition}

\begin{proof}

For a $g\in G$, the right action of $T_g$ on the space $(\text{ind}^G _{KZ} \sigma)^{I_1}$ is given by:

\begin{equation}\label{formula of T_g}
f\mid T_g= \sum_{i\in I_1 / (I_1\cap g^{-1} I_1 g)} ig^{-1}\cdot f,
\end{equation}
for an $f\in (\text{ind}^G _{KZ} \sigma)^{I_1}$. In this argument, we will take $g$ as $\gamma$ and $\omega_1$.

We treat $(1)$ at first. As $\gamma$ normalizes $I_1$, the sum above in this case only contains a single term, and $(1)$ follows by an argument similar to that of $(2)$ of Proposition \ref{the action of Iwahori}.

Now we deal with $(2)$. The above formula \eqref{formula of T_g} specifies to:
\begin{center}
$f_{\omega, \vec{a}} \mid T_{\omega_1} =\sum_{i\in I_1 / I_1 \cap \omega_1 I_1 \omega_1} i\omega_1 \cdot f_{\omega, \vec{a}}$
\end{center}
Note first the identification:
\begin{equation}\label{representatives for I/omega I omega}
I_1 / I_1 \cap \omega_1 I_1 \omega_1 \cong \begin{pmatrix} 1  & k_F & 0 \\ 0  & 1 & 0\\
0 & 0 & 1_{N-2}
\end{pmatrix}.
\end{equation}

Case $1)$.~  $a_1 > a_2$. Formally, we have a coset decomposition
\begin{center}
$KZ t_{\vec{a}} I_1 = \cup_{i \in (t^{-1}_{\vec{a}} K t_{\vec{a}} \cap I_1)\backslash I_1} KZ t_{\vec{a}} i$
\end{center}
In the case $a_1 > a_2$, a set of representatives for $(t^{-1}_{\vec{a}} K t_{\vec{a}} \cap I_1)\backslash I_1$ can be taken so that $i$ has zero $(1,2)$-entry and its $(2,1)$-entry goes through $\mathfrak{p}_F /\mathfrak{p}_F ^{a_1 -a_2}$. Hence for all such $i$, we have $\omega_1 i \omega_1 \in I_1$.

By the proceeding remark, we see immediately the function $f_{\omega, \vec{a}} \mid T_{\omega_1}$ is supported on $KZ t _{\omega_1\cdot \vec{a}} I_1$. To compute
\begin{center}
$\sum_{i \in I_1 /I_1 \cap \omega_1 I_1 \omega_1} f_{\omega, \vec{a}} (t _{\omega_1\cdot \vec{a}}\cdot i \omega_1),$
\end{center}
we separate the sum into two parts:
\begin{center}
$\sum_{i \in (I_1 \setminus I_1\cap \omega_1 I_1 \omega_1) /I_1 \cap \omega_1 I_1 \omega_1} f_{\omega, \vec{a}} (t _{\omega_1\cdot \vec{a}}\cdot i \omega_1) + \omega_1 f_{\omega, \vec{a}} (t_{\vec{a}})$
\end{center}
The first sum vanishes: we claim that $t _{\omega_1\cdot \vec{a}}\cdot i \omega_1 \notin KZ t _{\vec{a}} I_1$ for any $i \in I_1 \setminus I_1 \cap \omega_1 I_1 \omega_1$.

It remains to verify the claim. Recall the following identity:
\begin{center}
$\begin{pmatrix} 0  & 1  \\ 1  & 0 \end{pmatrix} \begin{pmatrix} 1  & t \\ 0  & 1 \end{pmatrix}= \begin{pmatrix} 1  & t^{-1} \\ 0  & 1 \end{pmatrix}\begin{pmatrix} -t^{-1}  & 0  \\ 0  & t \end{pmatrix}\begin{pmatrix} 1  & 0  \\ t^{-1}  & 1 \end{pmatrix}$
\end{center}
for $t \neq 0$. Using the identification \eqref{representatives for I/omega I omega} and the above identity, we can indeed prove
\begin{center}
$t _{\omega_1\cdot \vec{a}}\cdot i \omega_1 \in KZ t _{\omega_1\cdot \vec{a}} I_1$
\end{center}
for all such $i$. Note that it is this place the condition $a_1 > a_2$ plays a role. As $t_{\vec{a}} \neq t _{\omega_1\cdot \vec{a}}$, the claim follows.

Now the last term gives us $\omega_1 \omega^{-1} v_0$. As we already know $\omega\cdot \vec{a} \in S_{\omega\omega_1}$ (Lemma \ref{combin lemma}), we conclude that $f_{\omega, \vec{a}} \mid T_{\omega_1} = f_{\omega\omega_1, \omega_1\cdot \vec{a}}$.

\medskip
Case $2)$.~  $a_1 < a_2$.  When $N=3$, the situation $a_1 < a_2$ happens only for $\omega\in \{ Id, \omega_2 \omega_1, \omega^2 _2 \}$. Explicitly, a set of representatives for $(t^{-1}_{\vec{a}} K t_{\vec{a}} \cap I_1)\backslash I_1$ can be taken as follows:
\begin{center}
for $\omega= Id, a_1 < a_2 \leq 0$, $\begin{pmatrix} 1  & \mathfrak{o}_F /\mathfrak{p}_F ^{a_2 -a_1} & \mathfrak{o}_F /\mathfrak{p}_F ^{-a_1} \\ 0  & 1 & \mathfrak{o}_F /\mathfrak{p}_F ^{-a_2}\\
0 & 0 & 1
\end{pmatrix}$;
\end{center}

\begin{center}
for $\omega= \omega_2 \omega_1, a_1\leq 0, a_2 \geq 1$, $\begin{pmatrix} 1  & \mathfrak{o}_F /\mathfrak{p}_F ^{a_2 -a_1} & \mathfrak{o}_F /\mathfrak{p}_F ^{-a_1} \\ 0  & 1 & 0\\
0 & \mathfrak{p}_F /\mathfrak{p}_F ^{a_2} & 1
\end{pmatrix}$;
\end{center}
\begin{center}
for $\omega^2 _2, 1\leq a_1 < a_2$, $\begin{pmatrix} 1  & \mathfrak{o}_F /\mathfrak{p}_F ^{a_2 -a_1} & 0 \\ 0  & 1 & 0\\
\mathfrak{p}_F /\mathfrak{p}_F ^{a_1} & \mathfrak{p}_F /\mathfrak{p}_F ^{a_2} & 1
\end{pmatrix}$.
\end{center}

In every case, for an $i$ as above, we denote its $(12)$-entry by $t$. We can check that
\begin{center}
$KZ t_{\vec{a}} i \omega_1 I_1 =KZ t_{\vec{a}} u(t) \omega_1 I_1$,
\end{center}
where $u(t)$ is the matrix
\begin{center}
$\begin{pmatrix} 1  & t  & 0 \\ 0  & 1 & 0\\
0 & 0 & 1
\end{pmatrix}$.
\end{center}
Here, we note that $i$ can be written as $u(t)\cdot i'$ for some $i' \in I_1$ such that $\omega_1 i' \omega_1 \in I_1$.

Now for $t\in \mathfrak{p}_F$, we see the double coset is $KZ t_{\omega_1 \cdot \vec{a}} I_1$. For $t\in \mathfrak{o}_F \setminus \mathfrak{p}_F$, to verify the claim in Case $1)$ we have indeed shown
\begin{center}
$t_{\vec{a}} u(t)\omega_1 \in KZ t_{\vec{a}} I_1$.
\end{center}

In all, we see the function $f_{\omega, \vec{a}} \mid T_{\omega_1}$ is supported on $KZ t _{\omega_1\cdot \vec{a}} I_1 \cup KZ t_{\vec{a}} I_1$. Thus, it is a linear combination of $f_{\omega, \vec{a}}$ and $f_{\omega \omega_1, \omega_1 \cdot \vec{a}}$ (Proposition \ref{the pro-p iwahori invariants}). The statement in this case then follows by applying the quadratic relation $T^2 _{\omega_1}= -T_{\omega_1}$ and case $1)$.

\medskip

Case $3)$.~ $N=3, a_1 =a_2$. In this case, we can first check that
\begin{center}
$KZ t_{\vec{a}} I_1 \omega_1 I_1 = KZ t_{\vec{a}} I_1$,
\end{center}
where, a set of representatives for $(t^{-1}_{\vec{a}} K t_{\vec{a}} \cap I_1)\backslash I_1$ can be chosen such that $\omega_1 i \omega_1 \in I_1$. The function $f_{\omega, \vec{a}} \mid T_{\omega_1}$ is thus proportional to $f_{\omega, \vec{a}}$, say $c \cdot f_{\omega, \vec{a}}$ for some $c$. By the quadratic relation of $T_{\omega_1}$, $c$ can only be $0$ or $-1$. However, the exact value of $c$ depends on the weight $\sigma$ and $\omega$, and we denote $c$ by $c_{\sigma, \omega}$. Note in the current case $\omega \in \{ Id, \omega^2 _2\}$. We record the equation which determines $c_{\sigma, \omega}$:
\begin{center}
$\sum_{t\in k_F}~ \begin{pmatrix} 1  & t & 0 \\ 0  & 1 & 0\\
0 & 0 & 1
\end{pmatrix} \omega_1 \omega^{-1} v_0 = c_{\sigma, \omega} \omega^{-1} v_0.$
\end{center}

The argument for the proposition is done.
\end{proof}

\begin{remark}
The condition $W(\chi_\sigma)= W_0$ plays only the role that the two operators in consideration are well-defined, and it is not used in the argument.
\end{remark}

\emph{From now on, without causing confusion we will sometimes write a function of the form $c\cdot f_{\omega, \vec{a}}$ as $c\cdot \vec{a}$.}

\begin{proposition}\label{proper in the Iwahori--Hecke case}
Assume $N=3$ and $W(\chi_\sigma)=W_0$. We have:

$(1)$.~For $(a_1, a_2)\in S_{Id}$ satisfying $ a_1 \leq a_2 \leq 0$, we have

\begin{center}
$(a_1 -1, a_2)= (a_1, a_2)\mid (T_\gamma T_{\omega_1})^2 = (a_1, a_2)\mid T_{t_{(-1, 0)}}$;

$(a_1 -1, a_2-1)= (a_1, a_2)\mid (T _\gamma T_{\omega_1} T_\gamma)^2 =(a_1, a_2)\mid T_{t_{(-1, -1)}}$.
\end{center}

$(2)$.~For $(a_1, a_2) \in S_{\omega_1}$ satisfying $a_2 +1 \leq a_1 \leq 0$, we have:
\begin{center}
$(a_1, a_2 -1)= (a_1, a_2)\mid (T_{\omega_1}T_\gamma)^2 =(a_1, a_2)\mid T_{t_{(0, -1)}}$;

$(a_1-1, a_2 -1)= (a_1, a_2)\mid (T_\gamma T_{\omega_1}T_\gamma)^2=(a_1, a_2)\mid T_{t_{(-1, -1)}}$.
\end{center}

$(3)$.~For $(a_1, a_2) \in S_{\omega_1 \omega_2}$ satisfying $a_1 \geq a_2 +1 \geq 2$, we have
\begin{center}
$(a_1 +1, a_2)= (a_1, a_2) \mid (T_{\omega_1}T^2 _\gamma)^2=(a_1, a_2)\mid T_{t_{(1, 0)}}$;

$(a_1 +1, a_2 +1)= (a_1, a_2) \mid (T^2 _\gamma T_{\omega_1} T^2 _\gamma)^2=(a_1, a_2)\mid T_{t_{(1, 1)}}$.
\end{center}

$(4)$.~For $(a_1, a_2) \in S_{\omega_2 \omega_1}$ satisfying $a_1 \leq 0, a_2 \geq 1$, we have
\begin{center}
$(a_1 -1, a_2)= (a_1, a_2) \mid (T_\gamma T_{\omega_1})^2 =(a_1, a_2)\mid T_{t_{(-1, 0)}}$;

$(a_1, a_2 +1)= (a_1, a_2) \mid (T^2 _\gamma T_{\omega_1})^2 =(a_1, a_2)\mid T_{t_{(0, 1)}}$.
\end{center}

$(5)$.~For $(a_1, a_2) \in S_{\omega_2}$ satisfying $a_1 \geq 1, a_2 \leq 0$, we have:
\begin{center}
$(a_1 +1, a_2)= (a_1, a_2) \mid (T_{\omega_1}T^2 _\gamma)^2 =(a_1, a_2)\mid T_{t_{(1, 0)}}$;

$(a_1, a_2 -1) = (a_1, a_2)\mid (T_{\omega_1}T_\gamma)^2 =(a_1, a_2)\mid T_{t_{(0, -1)}}$.
\end{center}

$(6)$.~For $(a_1, a_2) \in S_{\omega^2 _2}$ satisfying $a_2 \geq a_1 \geq 1$, we have:
\begin{center}
$(a_1, a_2 +1) = (a_1, a_2) \mid (T^2 _\gamma T_{\omega_1})^2 =(a_1, a_2)\mid T_{t_{(0, 1)}}$;

$(a_1 +1, a_2 +1) = (a_1, a_2) \mid (T^2 _\gamma T_{\omega_1} T^2 _\gamma)^2 =(a_1, a_2)\mid T_{t_{(1, 1)}}$.

\end{center}

\end{proposition}

\begin{proof}
In every item listed above, the assertion for the first equality follows from $(1)$ and Case $1)$ of Proposition \ref{action of IH: trivial case}. In the process, we see the operator in the statement is unique (in terms of $T_\gamma$ and $T_{\omega_1}$). On the other hand, one can check that it equals a $T_g$ for some $g\in G$. This can be seen by considering the braid relations on the extended Weyl group. Indeed, in every case, the element $g$ is a diagonal matrix of the form $\varpi t_{\vec{a}}$ for some $\vec{a} \in \mathbb{Z}^2$, so by our definition the operator is nothing but $T_{t_{\vec{a}}}$. This gives the second equality.
\end{proof}

\subsection{The semi-regular case for $N=3$}

Suppose $\chi_\sigma$ is semi-regular, i.e., the group $W(\chi_\sigma)$ is of order $2$. We assume $\chi_\sigma$ is of the form $1\otimes 1\otimes \eta$ for some non-trivial character $\eta$, whence $W(\chi_\sigma)= \{Id, \omega_1\}$. In this case, Corollary \ref{basis of isotypic} says that $(\text{ind}^G _{KZ} \sigma)^{IZ, \chi_\sigma}$ has a basis $\mathcal{V}_{Id}\cup \mathcal{V}_{\omega_1}$, where $\mathcal{V}_{Id}= \{f_{Id, \vec{a}}\mid \vec{a}\in S_{Id}\}$ and $\mathcal{V}_{\omega_1}= \{f_{\omega_1 , \vec{a}}\mid \vec{a}\in S_{\omega_1}\}$. We will abuse the same notation to denote the subspace spanned by them.

\begin{proposition}\label{action of IH: semi-regular case}
For $\vec{a} \in S_{Id} \cup S_{\omega_1}$, we have the following:

$(1)$.~$\vec{a}\mid T_{\omega_1}= \begin{cases} \omega_1 \cdot\vec{a}, ~~~~~~\text{if}~a_1 >a_2; \\

c \cdot \vec{a}, ~~~~~~~~\text{if}~a_1= a_2; \\

-\vec{a}, ~~~~~~~~~~\text{if}~a_1< a_2. \\
\end{cases}$\\
Here, $\omega_1 \cdot\vec{a}= (a_2, a_1)$, and $c$ is a constant related to $\sigma$.

$(2)$.~For $(a_1, a_2) \in S_{Id} \cup S_{\omega_1}$, $(a_1, a_2) \mid T_{\omega_1\cdot t_{(0, -1)}}= (a_2, a_1 -1)$.

$(3)$.~ $T_{\omega_1 t_{(1, 0)}}$ vanishes on the whole space $(\textnormal{ind}^G _{KZ} \sigma)^{IZ, \chi_\sigma}$.

\end{proposition}

\begin{proof}
In the current case, the operator $T_{\omega_1}$ is well-defined (Corollary \ref{basis of I-W}). The argument of $(2)$ of Proposition \ref{action of IH: trivial case} works here, as that does not involve the specific information of the weight in consideration. Note that $\omega$ here lies in $\{Id, \omega_1\}$. So the first case only happens for $\omega =\omega_1$, and the other two cases happen for $\omega =Id$.

For $(2)$, we check first that the function $f_{\omega, \vec{a}} \mid T_{\omega_1\cdot t_{(0, -1)}}$ is supported on $KZ t_{\omega_1 \cdot \vec{a} +(0, -1)} I_1= KZ t_{(a_2, a_1 -1)}I_1$. Write $\omega_1 t_{0, -1}$ as $g$ for a moment. We compute that
\begin{center}
$f_{\omega, \vec{a}} \mid T_{\omega_1\cdot t_{(0, -1)}} (t_{(a_2, a_1 -1)})=\sum_{i\in I_1/ I_1 \cap g^{-1} I_1 g} f_{\omega, \vec{a}}(t_{(a_2, a_1 -1)} i\cdot t_{(0, 1)}\omega_1)$.
\end{center}
We note the identification
\begin{center}
$I_1/ I_1 \cap g^{-1} I_1 g \cong \begin{pmatrix} 1  & 0 & 0 \\ 0  & 1 & k_F\\
0 & 0 & 1
\end{pmatrix}.$
\end{center}

It remains to check the sum above is equal to $\omega_1 \omega^{-1} v_0$. Such a value is contributed by the term $i=Id$. For $i \neq Id$, we claim that
\begin{center}
$t_{(a_2, a_1 -1)} i\cdot t_{(0, 1)}\omega_1 \notin KZ t_{\vec{a}} I_1$.
\end{center}
This can be verified by combing the identification above with argument similar to that of Proposition \ref{action of IH: trivial case}. Indeed, we have
\begin{center}
$t_{(a_2, a_1 -1)} i\cdot t_{(0, 1)}\omega_1 \in KZ t_{(2-a_1, 1+a_2 -a_1)}I_1$,
\end{center}
where we note $a_1  \leq 0$.

For $(3)$, by estimating its support we firstly show $\vec{a} \mid T_{\omega_1 t_{(1, 0)}}$ lies in $\mathcal{V}_{\omega_1}$. Then the assertion follows from $(2)$ and $T_{\omega_1 t_{(1, 0)}}\cdot T_{\omega_1\cdot t_{(0, -1)}} =0$ (Proposition \ref{structure of Iwahori-Hecke: semi-regular}).
\end{proof}

\begin{proposition}\label{proper in the semi-regular case}
We have

$(1)$.~For $(a_1, a_2)\in S_{Id}$,
\begin{center}
$(a_1 -1, a_2)=(a_1, a_2)\mid T_{\omega_1 \cdot t_{(0, -1)}} T_{\omega_1}=(a_1, a_2)\mid T_{t_{(-1, 0)}}$,

$(a_1 -1, a_2 -1)= (a_1, a_2)\mid T^2 _{\omega_1 \cdot t_{(0, -1)}}=(a_1, a_2)\mid T_{t_{(-1, -1)}}$.
\end{center}

$(2)$.~ ~For $(a_1, a_2)\in S_{\omega_1}$,
\begin{center}
$(a_1, a_2 -1)= (a_1, a_2) \mid T_{\omega_1} T_{\omega_1 \cdot t_{(0, -1)}}= (a_1, a_2) \mid T_{t_{(0, -1)}}$,

$(a_1 -1, a_2 -1)= (a_1, a_2)\mid T^2 _{\omega_1 \cdot t_{(0, -1)}}=(a_1, a_2) \mid T_{t_{(-1, -1)}} $.
\end{center}
\end{proposition}

\begin{proof}
In every item in the statement, the assertion for the first equality simply follow from last Proposition. The assertion for the second equality holds because the two operators coincide, e.g., $T_{\omega_1 \cdot t_{(0, -1)}} T_{\omega_1}= T_{t_{(-1, 0)}}$,  by considering the braid relation on the extended Weyl group.
\end{proof}

\subsection{The regular case for $N=3$}

Suppose $\chi_\sigma$ is regular, i.e., $W(\chi_\sigma)= Id$. In this case, we know from Corollary \ref{basis of I-W} that a basis of $(\text{ind}^G _{KZ} \sigma)^{IZ, \chi_\sigma}$ is $\{f_{Id, \vec{a}} \mid \vec{a} \in S_{Id}\}$. We have the following:
\begin{proposition}\label{proper in the regular case}
For $(a_1, a_2)\in S_{Id}$, we have

$(1)$.~ $(a_1, a_2)\mid T_{t_{(-1, 0)}} = (a_1 -1, a_2), ~(a_1, a_2)\mid T_{t_{(-1, -1)}}= (a_1 -1, a_2 -1)$.

$(2)$.~$(a_1, a_2) \mid \mathcal{T} =0$ for $\mathcal{T} \in \{T_{t_{(1, 0)}}, T_{t_{(0, 1)}}, T_{t_{(0, -1)}}, T_{t_{(1, 1)}}\}$.

\end{proposition}

\begin{proof}
We only prove $(a_1, a_2)\mid T_{t_{(-1, 0)}} = (a_1 -1, a_2)$ in detail. First we check that the function $f_{Id, \vec{a}}\mid T_{t_{(-1, 0)}}$ is supported on $KZ t_{(a_1 -1, a_2)} I_1$.

We compute
\begin{center}
$f_{Id, \vec{a}}\mid T_{t_{(-1, 0)}} (t_{(a_1 -1, a_2)}) = \sum_{i \in I_1 / I_1 \cap t_{(1, 0)} I_1 t_{(-1, 0)}}~ f_{Id, \vec{a}} (t_{(a_1 -1, a_2)} i t_{(1, 0)})$.
\end{center}
We note the identification
\begin{center}
$I_1 / I_1 \cap t_{(1, 0)} I_1 t_{(-1, 0)} \cong \begin{pmatrix} 1  & k_F & k_F \\ 0  & 1 & 0\\
0 & 0 & 1
\end{pmatrix}.$
\end{center}

For $i=Id$, we get $f_{Id, \vec{a}} (t_{\vec{a}})= v_0$. 

For $i \neq Id$, we claim that $t_{(a_1 -1, a_2)} i t_{(1, 0)} \notin KZ t_{\vec{a}} I_1$.

By the identification above, we write $i$ as
\begin{center}
$\begin{pmatrix} 1  & t_1 & t_2\\ 0  & 1 & 0\\
0 & 0 & 1
\end{pmatrix}$
\end{center}
for $t_1, t_2 \in k_F$.

When $t_1 \neq 0, t_2=0$, by a similar argument as that of Proposition \ref{action of IH: trivial case}, we have $t_{(a_1 -1, a_2)} i t_{(1, 0)} \in KZ t_{(a_2 +1, a_1 -1)} I_1$.

When $t_1\cdot t_2 \neq 0$, starting from the case above, we show first that $t_{(a_1 -1, a_2)} i t_{(1, 0)}$ shares the same double coset with
\begin{center}
$t_{(a_2 +1, a_1 -1)} \begin{pmatrix} 1  & 0 & \varpi^{-1} t_2\\ 0  & 1 & 0\\
0 & 0 & 1
\end{pmatrix}$.
\end{center}
Then by an argument similar to that of semi-regular case, we can check
\begin{center}
$t_{(a_2 +1, a_1 -1)} \begin{pmatrix} 1  & 0 & \varpi^{-1} t_2\\ 0  & 1 & 0\\
0 & 0 & 1
\end{pmatrix} \in KZ t_{(1- a_2, a_1 -1 -a_2)} I_1$.
\end{center}

When $t_1 =0, t_2 \neq 0$, we have already shown in the semi-regular case that
\begin{center}
$t_{(a_1 -1, a_2)} i t_{(1, 0)} \in KZ t_{(2- a_1, 1+a_2 -a_1)} I_1$.
\end{center}

In all cases, the condition $a_1 \leq a_2 \leq 0$ is essentially used in the argument. The claim is verified. The assertion $(a_1, a_2)\mid T_{t_{(-1, 0)}} = (a_1 -1, a_2)$ is proved.

\smallskip
Now the assertions in $(2)$ about $T_{(1, 0)}, T_{(1, 1)}, T_{(0, -1)}$ follows from what we have just proved and Proposition \ref{structure of Iwahori-Hecke: regular}.

\emph{The remaining two assertions can be proved by the same manner and the details are omitted.}
\end{proof}

\section{Proof of Theorem \ref{main result: intro}}

\begin{definition}
For an $\vec{a}\in S_\omega$, we say a vector $\vec{b}\in S_\omega$ is \textbf{proper} to $\vec{a}$ if
\begin{center}
$\omega t_{\vec{b}}\cdot I \cdot (\omega t_{\vec{b}})^{-1}\cap K \subseteq \omega t_{\vec{a}}\cdot I \cdot (\omega t_{\vec{a}})^{-1}\cap K$.
\end{center}
Furthermore, we say $\vec{b}$ is \textbf{strictly} proper to $\vec{a}$ if the containing above is strict.
\end{definition}

\smallskip

\begin{theorem}\label{main theorem for GL_3 in appendix}
Assume $N=3$. Let $\vec{a}\in S_{Id}$. If $\vec{a}^\ast$ is proper to $\vec{a}$ in $S_{Id}$, then there exists (not unique in general) an operator $\mathcal{T}\in \mathcal{H}(IZ, \chi_\sigma)$ such that
\begin{center}
$f_{Id, \vec{a}}\mid \mathcal{T}= f_{Id, \vec{a}^\ast}$.
\end{center}

\end{theorem}

\begin{proof}

For any vector $\vec{a}^\ast$ \emph{strictly proper} to $\vec{a}$ we can obtain the corresponding function $f_{Id, \vec{a}^\ast}$ by applying Proposition \ref{proper in the Iwahori--Hecke case},\ref{proper in the semi-regular case},\ref{proper in the regular case} repeatedly (but such process might not be unique anymore).
\end{proof}

Based on last theorem, we have the following corollary.
\begin{corollary}\label{GL_3 fails}
There are non-zero $\mathcal{H}(IZ, \chi_\sigma)$-submodules of $(\textnormal{ind}^G _{KZ} \sigma)^{IZ, \chi_\sigma}$ of infinite codimension.
\end{corollary}

\begin{proof}
Consider the submodule $M$ of $(\textnormal{ind}^G _{KZ} \sigma)^{IZ, \chi_\sigma}$ generated by a single function $f_{Id, \vec{a}}$, for an $\vec{a}=(a_1, a_2)\in S_{Id}$ satisfying $a_1 < a_2$. By Theorem \ref{main theorem for GL_3 in appendix}, the module $M$ contains the subspace $M'$ spanned by functions
\begin{center}
 $\{f_{Id, \vec{b}} \mid \vec{b}~\text{is~proper~to}~\vec{a} \}$
\end{center}
We claim that if a function $f \in M$ is supported in the cosets $\bigcup_{\vec{d}\in S_{Id}} KZ t_{\vec{d}} I_1$, then it lies in $M'$. Suppose that $f= f_{Id, \vec{a}} \mid \mathcal{T}$ for some $\mathcal{T} \in \mathcal{H}(IZ, \sigma)$.

We verify the claim case by case.

Case $a)$. $W(\chi_\sigma)= Id$. By Proposition \ref{structure of Iwahori-Hecke: regular} we write $\mathcal{T}$ as a polynomial of $T_{t_{(1, 0)}}, T_{t_{(-1, 0)}}, T_{t_{(0, 1)}}$ and $T_{t_{(0, -1)}}, T_{t_{(1, 1)}}, T_{t_{(-1, -1)}}$. By Proposition \ref{proper in the regular case}, we see $\mathcal{T}$ can be replaced by a polynomial of $T_{t_{(-1, 0)}}$ and $T_{t_{(-1, -1)}}$.

Case $b)$.~$W(\chi_\sigma)= \{Id, \omega_1\}$. We write $\mathcal{T}$ as a polynomial of $T_{\omega_1}$, $T_{\omega_1\cdot t_{(1, 0)}}$, and $T_{\omega_1\cdot t_{(0, -1)}}$ (Proposition \ref{structure of Iwahori-Hecke: semi-regular}). Assume $f$ is supported in the cosets $\bigcup_{\vec{d}\in S_{Id}} KZ t_{\vec{d}} I_1$. We can firstly eliminate from $\mathcal{T}$ any monomial with the operator $T_{\omega_1 \cdot t_{1, 0}}$ (Proposition \ref{action of IH: semi-regular case}). As a polynomial of $T_{\omega_1}$ acts as a scalar on $f_{Id, \vec{a}}$ (Proposition \ref{action of IH: semi-regular case}), we may assume any monomial of it does not appear in $\mathcal{T}$ \footnote{The same remark applies to Case $c)$.}. Then $\mathcal{T}$ can be substituted with a polynomial of $T_{\omega_1\cdot t_{(0, -1)}} T_{\omega_1}$ and $T^2 _{\omega_1\cdot t_{(0, -1)}}$, as anything else would force $f$ to be supported outside $\bigcup_{\vec{d}\in S_{Id}} KZ t_{\vec{d}} I_1$ (Proposition \ref{action of IH: semi-regular case}).

Case $c)$.~ $W(\chi_\sigma)= W_0$. We write $\mathcal{T}$ as a polynomial of $T_\gamma$ and $T_{\omega_1}$ (Proposition \ref{structure of Iwahori-Hecke: degenerate case}). If $f$ is supported in the cosets $\bigcup_{\vec{d}\in S_{Id}} KZ t_{\vec{d}} I_1$, then $\mathcal{T}$ can be replaced by a polynomial of $(T_\gamma T_{\omega_1})^2$ and $(T _\gamma T_{\omega_1} T_\gamma)^2$; otherwise, anything else would force $f$ to be supported outside $\bigcup_{\vec{d}\in S_{Id}} KZ t_{\vec{d}} I_1$ (Proposition \ref{action of IH: trivial case}). Note here we have assumed $a_1 < a_2$.

In every case, the claim follows by Theorem \ref{main theorem for GL_3 in appendix}.

We now assume $a_2 < -1$. Consider the subspace $M''$ of $(\textnormal{ind}^G _{KZ} \sigma)^{I_1}$ spanned by
 \begin{center}
$\{f_{Id, \vec{c}} \mid \vec{c}=(\ast, a_2 +1), \ast \leq a_2 +1\}$.
 \end{center}
The space $M''$ is infinitely dimensional, but by our claim above it has \emph{no} non-zero intersection with $M$. The argument is done.
\end{proof}

\section*{Acknowledgements}
A major part of this note was done when the author was a postdoc at Warwick Mathematics Institute (Leverhulme Trust RPG-2014-106) and Einstein Institute of Mathematics (ERC 669655), and he would like to thank both institutions for the hospitality.

\bibliographystyle{amsalpha}
\bibliography{new}

\providecommand{\bysame}{\leavevmode\hbox to3em{\hrulefill}\thinspace}
\providecommand{\MR}{\relax\ifhmode\unskip\space\fi MR }
\providecommand{\MRhref}[2]{%
  \href{http://www.ams.org/mathscinet-getitem?mr=#1}{#2}
}
\providecommand{\href}[2]{#2}
\begin{thebibliography}{AHHV17}

\bibitem[AHHV17]{AHHV17b}
N.~Abe, G.~Henniart, F.~Herzig, and M.-F. Vign\'eras, \emph{Questions on {${\rm
  mod}\, p$} representations of {$p$}-adic reductive groups},
  https://arxiv.org/abs/1703.02063, 2017.

\bibitem[BL94]{B-L94}
Laure Barthel and Ron Livn{\'e}, \emph{Irreducible modular representations of
  {${\rm GL}_2$} of a local field}, Duke Math. J. \textbf{75} (1994), no.~2,
  261--292. \MR{1290194 (95g:22030)}

\bibitem[BL95]{B-L95}
\bysame, \emph{Modular representations of {${\rm GL}_2$} of a local field: the
  ordinary, unramified case}, J. Number Theory \textbf{55} (1995), no.~1,
  1--27. \MR{1361556 (96m:22036)}

\bibitem[CE04]{C-E2004}
Marc Cabanes and Michel Enguehard, \emph{Representation theory of finite
  reductive groups}, New Mathematical Monographs, vol.~1, Cambridge University
  Press, Cambridge, 2004. \MR{2057756 (2005g:20067)}

\bibitem[Oll06]{Oll06}
Rachel Ollivier, \emph{Modules simples en caract\'eristique {$p$} de
  l'alg\`ebre de {H}ecke du pro-{$p$}-{I}wahori de {${\rm GL}_3(F)$}}, J.
  Algebra \textbf{304} (2006), no.~1, 1--38. \MR{2255819 (2007j:22030)}

\bibitem[Oll15]{Oll2015}
\bysame, \emph{An inverse {S}atake isomorphism in characteristic {$p$}},
  Selecta Math. (N.S.) \textbf{21} (2015), no.~3, 727--761. \MR{3366919}

\bibitem[Vig05]{Vig2005}
Marie-France Vign\'eras, \emph{Pro-{$p$}-{I}wahori {H}ecke ring and
  supersingular {$\overline{\bold F}_p$}-representations}, Math. Ann.
  \textbf{331} (2005), no.~3, 523--556. \MR{2122539}

\bibitem[Xu18]{X2018}
Peng Xu, \emph{Hecke eigenvalues in {$p$}-modular representations of unramified
  {${\rm U}(2,1)$}}, Preprint, 2018.

\end{thebibliography}

\end{document}